\documentclass[12pt,a4paper]{amsart}
\usepackage[a4paper, footskip=0.5in,  headheight = 0.5in, top=1.25in, bottom=1.25in,  right=1in,  left=1in]{geometry}
\usepackage{macros}

\title[Equitable Coloring of Large Bipartite Graphs]{Equitable Coloring of Large Bipartite Graphs}

\author{Amir Nikabadi}
\address{IT University of Copenhagen, Denmark}
\thanks{Supported by the Independent Research Fund Denmark (DFF), grant agreement number 2098-00012B}
\date{\today}

\begin{document}

\maketitle

\begin{abstract}
For a graph $G$, the \emph{equitable chromatic number} of $G$, denoted by $\chi_e(G)$, is the smallest integer $k$ such that $G$ admits a proper $k$-coloring whose color classes differ in size by at most one. We prove that for every $\zeta>41/2$, there exists a constant $c=c(\zeta)\in\mathbb{N}$ such that every bipartite graph $G$ with maximum degree $\Delta(G)\ge c$ and $|V(G)|\ge \zeta\Delta(G)$ satisfies
$\chi_e(G)\le \left\lceil\Delta(G)/2\right\rceil+1$. 
For general bipartite graphs, the star $K_{1,\Delta}$ witnesses that no upper bound
depending only on $\Delta(G)$ can have leading term smaller than
$\Delta(G)/2$.
Our proof yields an $O(|V(G)|^2)$-time algorithm for constructing such a coloring.
\end{abstract}

\section{Introduction}
We use $\mathbb{N}$ to denote the set of positive integers.
We let $[n] := \{1, \dots , n \}$ for every $n \in \mathbb{N}$.
Graphs in this paper have finite vertex sets, no loops, and no parallel edges. Given a graph $G=(V(G),E(G))$, the maximum degree of $G$ is denoted by $\Delta(G)$.
An \emph{independent set} in $G$ is a set of pairwise non-adjacent vertices.
For disjoint $X, Y \subseteq V(G)$, we say that $X$ is \emph{complete} to $Y$ if every vertex in $X$ is adjacent to every vertex in $Y$, and $X$ is \emph{anticomplete} to $Y$ if there
are no edges between $X$ and $Y$.
A proper $k$-coloring of $G$ is \emph{equitable} if the sizes of any two color classes differ by at most one. The smallest integer $k$ such that $G$ admits an equitable $k$-coloring is called the \emph{equitable chromatic number} of $G$, and is denoted by $\chi_e(G)$. Clearly, for every graph $G$, we have $\chi(G)\le \chi_e(G)$.

Equitable coloring has a long history.
A celebrated result of Hajnal and Szemerédi~\cite{hajnal1970proof} settled a conjecture of Erdős~\cite{erdos1964problem} by showing that every graph $G$ with maximum degree at most $\Delta$ admits an equitable $(\Delta+1)$-coloring. The bound in the Hajnal--Szemerédi theorem is sharp for complete graphs $K_{\Delta+1}$ and for complete bipartite graphs $K_{\Delta,\Delta}$ when $\Delta$ is odd.

In the ordinary coloring setting, Brooks' theorem gives a substantial strengthening of the trivial bound $\chi(G)\le \Delta(G)+1$, showing that every connected graph $G$ satisfies $\chi(G)\le \Delta(G)$ except for the usual complete graph and odd cycle obstructions. The corresponding equitable analogue remains open in general. In particular, the following strengthening of the Hajnal--Szemerédi theorem was conjectured by Chen, Lih, and Wu~\cite{chen1994equitable}.

\begin{conjecture}[Chen, Lih, and Wu~\cite{chen1994equitable}]\label{conj:eq-D}
Every connected graph $G$ with maximum degree $\Delta\ge 2$ has an equitable coloring with $\Delta$ colors, except when $G$ is a complete graph, or an odd cycle, or $\Delta$ is odd and $G=K_{\Delta,\Delta}$.
\end{conjecture}

Conjecture~\ref{conj:eq-D} remains open in general. Chen, Lih, and Wu~\cite{chen1994equitable} proved it for graphs with $\Delta(G)\le 3$ and for graphs with $\Delta(G)\ge |V(G)|/2$. For trees, it was proved by Chen and Lih~\cite{chen1994-treeequitable}. Lih and Wu~\cite{lih1996equitable} proved it for connected bipartite graphs.

The bipartite setting, however, exhibits extremal behavior that differs substantially from the general $\Delta$-coloring phenomenon. In particular, the star $K_{1,\Delta}$ satisfies $\chi_e(K_{1,\Delta})=\lceil \Delta/2\rceil+1$, so any upper bound for bipartite graphs must, in general, have leading term at least $\Delta/2$. It is therefore natural to ask whether every bipartite graph satisfies an upper bound of the form $\chi_e(G)\le \lceil \Delta(G)/2\rceil+O(1)$.

Let us mention two results in this vein.
Bollobás and Guy~\cite{bollobas1983equitable} proved that a tree $T$ is
equitably 3-colorable if $|T| \geq 3\Delta(T)-8$ or $|T| = 3\Delta(T)-10$ and provided an algorithm for
producing the coloring. This result was extended to all $k \geq 2$ and to all forests by Chen and Lih~\cite{chen1994-treeequitable}.
In a different regime, a related result of Lih and Wu~\cite{lih1996equitable} shows that if $G=G(A,B)$ is a connected bipartite graph with $|A|=a\ge b=|B|$ and
$|E(G)|<\left\lfloor \frac{a}{b+1}\right\rfloor (a-b)+2a,$
then
$\chi_e(G)\le \left\lceil \frac{a}{b+1}\right\rceil+1.$
Thus their result is governed by the imbalance between the two sides of the bipartition together with an explicit sparsity condition on the number of edges, rather than by $\Delta(G)$ alone.

As far as we know (and curiously enough), no upper bound of the form
$\chi_e(G)\le \left\lceil \frac{\Delta(G)}{2}\right\rceil+O(1)$, stated solely in terms of $\Delta(G)$, has been established for bipartite graphs. The purpose of this paper is to prove such a bound for bipartite graphs whose order is sufficiently large compared to their maximum degree. In particular, we prove the following. 

\begin{theorem}\label{thm:main}
For every $\zeta>\frac{41}{2}$, there exists $c_{\ref{thm:main}}=c_{\ref{thm:main}}(\zeta)\in\mathbb{N}$ such that every bipartite graph $G$ with $\Delta(G)\ge c_{\ref{thm:main}}$ and $|V(G)|\ge \zeta\Delta(G)$ satisfies
$
\chi_e(G)\le \left\lceil \Delta(G)/2\right\rceil+1.
$
\end{theorem}

\Cref{thm:main} complements the results of Bollobás and Guy~\cite{bollobas1983equitable} and Lih and Wu~\cite{lih1996equitable} by showing that, under a linear lower bound on the order, every bipartite graph satisfies an upper bound of the form
$\chi_e(G)\le \left\lceil \Delta(G)/2\right\rceil+1.$
Moreover, our result yields an $O(|V(G)|^2)$-time algorithm for constructing such a coloring.

\smallskip
\noindent
\subsection*{Note} We emphasize that the extremal example $K_{1,\Delta}$ does not satisfy the
large-order hypothesis in Theorem~\ref{thm:main}. Thus our result should not
be interpreted as sharp for the class of bipartite graphs with
$|V(G)|\ge \zeta\Delta(G)$. Rather, the star shows that the term
$\Delta/2$ is unavoidable for any statement covering all bipartite graphs
and depending only on $\Delta$. It remains an interesting open problem to
determine whether, for fixed $\zeta$, the equitable chromatic number of
bipartite graphs satisfying $|V(G)|\ge \zeta\Delta(G)$ can still grow
linearly in $\Delta$, or whether a substantially smaller bound is possible.

\medskip

From now on, unless stated otherwise, we assume that $G$ is a bipartite graph with bipartition $(A,B)$ chosen so that $\mathfrak{a}:=|A|\le |B|=:\mathfrak{b}$. We also write $\Delta:=\Delta(G)$, set $k:=\lceil \Delta/2\rceil+1$, and write $n:=|V(G)|=kq+r$ with $0\le r<k$.

\subsection*{Proof outline}
We briefly explain the main idea of the proof of~\Cref{thm:main}.
Since $n=kq+r$, an equitable $k$-coloring of $G$ is equivalent to a partition of $V(G)$ into $k$ independent sets, exactly $r$ of size $q+1$ and the remaining $k-r$ of size $q$.
The proof begins by writing $\mathfrak a=|A|$ in the form $\mathfrak a=xq+u+M$, where $x$ and $u$ are chosen so that $xq+u$ vertices of $A$ can already be covered by $x$ classes contained in $A$, exactly $u$ of size $q+1$, while the remainder $M$ is small. The existence of such a representation is given by Lemma~\ref{lem:normalization}.
One then applies Lemma~\ref{lem:construction}. First, the set of size $xq+u$ inside $A$ is covered by classes contained in $A$. The remaining $M$ vertices of $A$ are distributed among $t=\lfloor k/4\rfloor$ further classes, each of which is completed using vertices of $B$ so as to remain independent; these classes are then rebalanced to obtain the required numbers of classes of sizes $q$ and $q+1$. Finally, all uncovered vertices lie in $B$, and since $B$ is independent they can be partitioned directly into the remaining color classes.
Thus the proof of~\Cref{thm:main} reduces to verifying that, under the hypothesis $|V(G)|\ge \zeta\Delta(G)$, the conditions required in Lemma~\ref{lem:construction} are satisfied.

Let us also indicate the role of the choice $t=\lfloor k/4\rfloor$. The parameter $t$ is the number of further classes used to
deal with the remaining $M$ vertices of $A$. Taking $t$ larger gives more
classes among which these vertices can be distributed, but it also requires
more vertices of $B$ to complete those classes and leaves fewer classes for
the final partition of the uncovered part of $B$. Thus $t$ must balance
these two competing requirements. The choice $t=\lfloor k/4\rfloor$ is large enough to handle the
remaining $M$ vertices of $A$, while still leaving enough vertices and
enough classes for the final step. The numerical constant in
\Cref{thm:main} comes from making this balancing argument precise, including
floor effects and the estimate $M<q+k$; it is not intended to represent a
genuine extremal threshold.

\section{Partitioning Lemmas}

We begin with some definitions and an observation (whose proof is easy and we omit). 
Let $X\subseteq V(G)$ and $q\geq 1$. A \emph{$(q,q+1)$-cover} of $X$ is a family of pairwise disjoint independent sets whose union is $X$ and each of whose members has size either $q$ or $q+1$.
An \emph{$A$-pure class} is an independent set contained in $A$, a \emph{$B$-pure class} is an independent set contained in $B$, and a \emph{mixed class} is an independent set intersecting both $A$ and $B$.

\begin{observation}\label{lem:split}
Let \(M,H\) be nonnegative integers, and let \(t\) be a positive integer. If \(M\le tH\), then there exist integers \(s_1,\dots,s_t\) such that
$\sum_{i=1}^t s_i=M$, $ 0\le s_i\le H$ for all $i\in [t]$,
and moreover \(|s_i-s_j|\le 1\) for all \(i,j\in [t]\).
\end{observation}


The next lemma gives a three-step covering procedure, which is at the heart of the proof of~\Cref{thm:main}. 

\begin{lemma}\label{lem:construction}
Assume $\Delta\ge 2$, $q\ge 1$, and $0\leq r<k$. Let $t\in \mathbb{N}\cup \{0\}$ and $x,u$ be nonnegative integers. Let $S\subseteq A$ satisfy $|S|=xq+u$, where $0\le u\le x$ and $u\le r$. Put $M:=\mathfrak{a}-|S|$ and let
\[
H:=\min\left\{q,\left\lfloor \frac{\mathfrak{b}-t(q+1)}{\Delta-1}\right\rfloor\right\}.
\]
Suppose that $r-u\le k-x$, $t\le k-x$, $H\ge 0$, and $tH\ge M$.
Then there exists a $(q,q+1)$-cover $\mathcal C$ of $V(G)$ consisting of $k$ pairwise disjoint independent sets such that exactly $r$ members of $\mathcal C$ have size $q+1$.
\end{lemma}

\begin{proof}
We construct the desired family $\mathcal C$ in three steps (see~\Cref{fig:chair-bull-hedg}).
First, we start by covering $S$. Since $|S|=xq+u=(x-u)q+u(q+1)$ and $S\subseteq A$, the set $S$ admits a $(q,q+1)$-cover $\mathcal C_A$ by $x$ $A$-pure classes, exactly $u$ of size $q+1$ and the remaining $x-u$ of size $q$.
If $t=0$, then $tH\ge M$ implies $M=0$, so $A\setminus S=\emptyset$. In this case let $\mathcal C_M:=\emptyset$. Set $y:=k-x-t$. Since
$\mathfrak a=|S|=xq+u,$
we have
$\mathfrak b=n-\mathfrak a=(kq+r)-(xq+u)=(k-x)q+(r-u)=yq+(r-u).$
Also $0\le r-u\le k-x=y$. Hence $B$ admits a $(q,q+1)$-cover $\mathcal C_B$ by $y$ $B$-pure classes, exactly $r-u$ of size $q+1$ and the remaining $y-r+u$ of size $q$.
Then $\mathcal C:=\mathcal C_A\cup \mathcal C_M\cup \mathcal C_B$ is a family of $x+y=k$ pairwise disjoint independent sets covering $V(G)$. Moreover, exactly $u+(r-u)=r$ members of $\mathcal C$ have size $q+1$. Thus $\mathcal C$ is the required $(q,q+1)$-cover of $V(G)$.
Henceforth, we may assume that $t\ge 1$.

We cover $A\setminus S$ as follows. Since $H\ge 0$ and $tH\ge M$, Observation~\ref{lem:split} gives integers $s_1,\dots,s_t$ such that $0\le s_i\le H$ for all $i\in [t]$ and $\sum_{i=1}^t s_i=M$. As $H\le q$, we have $s_i\le q$, and hence $q+1-s_i\ge 1$ for every $i\in [t]$.
Write $R:=A\setminus S$, so $|R|=M$. We now construct pairwise disjoint independent sets $C_1,\dots,C_t$, each of size $q+1$, such that $C_i=R_i\cup T_i$, where $R_i\subseteq R$, $T_i\subseteq B$, and $|R_i|=s_i$. Suppose that $C_1,\dots,C_{i-1}$ have already been chosen. Since $\sum_{j=1}^t s_j=M$, the number of vertices of $R$ not yet used is $M-\sum_{j=1}^{i-1}s_j=\sum_{j=i}^t s_j\ge s_i$, so we may choose $R_i\subseteq R$ with $|R_i|=s_i$.

At this stage, at most $(i-1)(q+1)$ vertices of $B$ have been used. Also every vertex of $R_i$ has degree at most $\Delta$, so $|N(R_i)|\le s_i\Delta$. Hence the number of unused vertices of $B$ with no neighbor in $R_i$ is at least $\mathfrak{b}-(i-1)(q+1)-s_i\Delta$. It therefore suffices to show that $\mathfrak{b}-(i-1)(q+1)-s_i\Delta\ge q+1-s_i$. But since $i\le t$, $s_i\le H$, and $H\le \left\lfloor \frac{\mathfrak{b}-t(q+1)}{\Delta-1}\right\rfloor$, we have
$
\mathfrak{b}-(i-1)(q+1)-s_i\Delta-(q+1-s_i)\ge
\mathfrak{b}-t(q+1)-H(\Delta-1)\ge 0.
$
Thus we may choose $T_i\subseteq B$ of size $q+1-s_i$ consisting of unused vertices anticomplete to $R_i$. Since $R_i\subseteq A$, $T_i\subseteq B$, and $T_i$ is anticomplete to $R_i$, the set
$
C_i:=R_i\cup T_i
$
is independent. Moreover, $C_i$ intersects $B$, and it is mixed whenever $s_i>0$ (possibly $B$-pure when $s_i=0$). Let
$
\mathcal C_M:=\{C_1,\dots,C_t\}.
$
The classes in $\mathcal C_M$ all initially have size $q+1$, and we now rebalance them.

Let $y:=k-x-t$. Since $t\le k-x$, we have $y\ge 0$. Also $0\le r-u\le k-x=y+t$. Choose
$e:=\max\{0,r-u-y\}.$
Indeed, the final $y$ classes contained in $B$ can account for at most $y$ classes of size $q+1$, so among the $t$ classes in $\mathcal C_M$ at least $r-u-y$ must remain of size $q+1$. Thus $e$ is chosen as the smallest feasible value. Then $0\le e\le t$ and $0\le r-u-e\le y$.
Keep exactly $e$ members of $\mathcal C_M$ unchanged, and from each of the remaining $t-e$ members delete one vertex of $B$. This preserves independence, because each class $C_i\in\mathcal C_M$ is already an independent set. After this step, the resulting family, which we still denote by $\mathcal C_M$, consists of $t$ pairwise disjoint independent sets, exactly $e$ of size $q+1$ and the remaining $t-e$ of size $q$.

Finally we cover the remaining vertices of $B$. Since $n=kq+r$ and $\mathfrak{a}=|S|+|R|=xq+u+M$, we have $\mathfrak{b}=(k-x)q+(r-u)-M$. The family $\mathcal C_M$ uses $tq+e-M$ vertices of $B$ (because before shrinking it used $\sum_{i=1}^t(q+1-s_i)=t(q+1)-M$ vertices of $B$, and we deleted one vertex from exactly $t-e$ of its members). Hence the number of vertices of $B$ still uncovered is
$\mathfrak{b}-(tq+e-M)=(k-x-t)q+(r-u-e)=yq+(r-u-e)$.
Since $0\le r-u-e\le y$, this can be written as $(r-u-e)(q+1)+(y-r+u+e)q$. As $B$ is independent, the remaining vertices of $B$ admit a $(q,q+1)$-cover $\mathcal C_B$ by $y$ $B$-pure classes, exactly $r-u-e$ of size $q+1$ and the remaining $y-r+u+e$ of size $q$.

Now let $\mathcal C:=\mathcal C_A\cup \mathcal C_M\cup \mathcal C_B$. Then $\mathcal C$ is a family of pairwise disjoint independent sets covering $V(G)$, and it has size $x+t+y=k$. Moreover, exactly $u+e+(r-u-e)=r$ members of $\mathcal C$ have size $q+1$, while every other member has size $q$. Thus $\mathcal C$ is the required $(q,q+1)$-cover of $V(G)$. This proves~\ref{lem:construction}.
\end{proof}

\begin{figure}[t]
    \centering
    \includegraphics[width=1\textwidth]{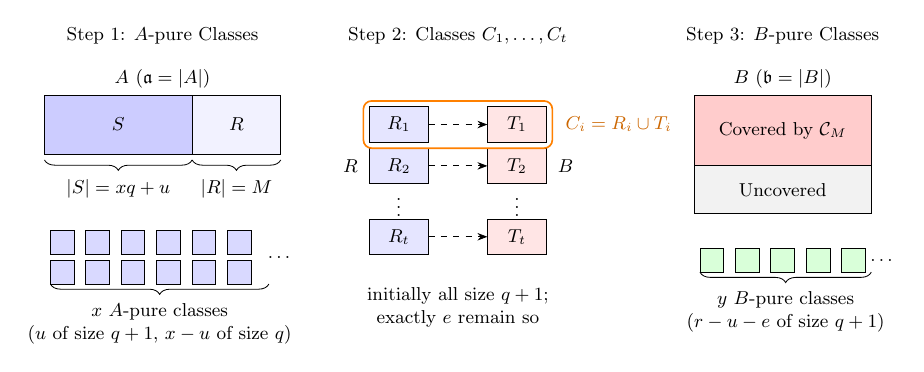}
\caption{The three-step covering procedure used in the proof of \Cref{lem:construction}.}
    \label{fig:chair-bull-hedg}
\end{figure}

For fixed integers $q,k,r$, we say that an integer $\beta$ \emph{admits a normalized form} if there exist integers $x,u,M$ such that
$\beta=xq+u+M$ where
\begin{itemize}
    \setlength\itemsep{0em}
    \item $0\le u\le x\le \lfloor k/2\rfloor$,
    \item $u\le r$, $r-u\le k-x$, and
    \item $0\le M<q+k$.
\end{itemize}

We next show that the parameters in Lemma~\ref{lem:construction} can always be chosen in normalized form.

\begin{lemma}\label{lem:normalization}
Let $q\ge 1$. Then $\mathfrak{a}$ admits a normalized form.
\end{lemma}

\begin{proof}
Choose
\[
x_0:=\min\left\{\left\lfloor \frac{\mathfrak{a}}{q}\right\rfloor,\left\lfloor \frac{k}{2}\right\rfloor\right\}
\quad\text{and}\quad
m_0:=\mathfrak{a}-x_0q.
\]
Thus $x_0$ is the largest integer satisfying $x_0q\le \mathfrak a$ and $x_0\le \lfloor k/2\rfloor$, and $m_0$ is the corresponding residual term. Let
\[
\ell_0:=\max\{0,x_0+r-k\}.
\]
Then $\ell_0$ is the smallest nonnegative integer $u$ such that $r-u\le k-x_0$ holds.
By construction, $0\le \ell_0\le x_0\le \left\lfloor \frac{k}{2}\right\rfloor$. We also have $\ell_0\le r$, since $\ell_0=\max\{0,x_0+r-k\}$ and $x_0\le k$, so $x_0+r-k\le r$.
We first note that $0\le m_0<q+k$. If $x_0=\lfloor \mathfrak{a}/q\rfloor$, then this is immediate. Otherwise $x_0=\lfloor k/2\rfloor$, and since $\mathfrak{a} \le n/2=(kq+r)/2$, we get $m_0=\mathfrak{a}-\left\lfloor \frac{k}{2}\right\rfloor q\le \frac{kq+r}{2}-\left\lfloor \frac{k}{2}\right\rfloor q\le \frac{q+r}{2}<q+k$.

If $m_0\ge \ell_0$, we take $x:=x_0$, $u:=\ell_0$, and $M:=m_0-\ell_0$. Then $\mathfrak{a}=xq+u+M$, and the inequalities $0\le u\le x\le \left\lfloor \frac{k}{2}\right\rfloor$, $u\le r$, and $0\le M<q+k$ are immediate. Moreover, by the choice of $\ell_0$, we also have $r-u\le k-x$. This gives the required form.

We may therefore assume that $m_0<\ell_0$. Choose
\[
d:=\left\lceil \frac{\ell_0-m_0}{q+1}\right\rceil,
\]
and set
\[
x:=x_0-d,\qquad u:=\ell_0-d,\qquad M:=m_0-\ell_0+d(q+1).
\]
Since $0<\ell_0-m_0\le \ell_0$, we have $1\le d\le \ell_0\le x_0$. Hence
$
0\le u\le x\le \left\lfloor \frac{k}{2}\right\rfloor,
$
and also $u\le \ell_0\le r$. Moreover,
$
\mathfrak{a}=(x_0-d)q+(\ell_0-d)+\bigl(m_0-\ell_0+d(q+1)\bigr)=xq+u+M.
$
Since here $\ell_0>0$, we have $\ell_0=x_0+r-k$, and therefore
$
r-u=r-(\ell_0-d)=k-x_0+d=k-x.
$

Finally, the definition of $d$ gives $(d-1)(q+1)<\ell_0-m_0\le d(q+1)$, so $0\le d(q+1)-(\ell_0-m_0)<q+1$. Since $M=d(q+1)-(\ell_0-m_0)$, it follows that $0\le M<q+1 \leq q+k$.

It follows in all cases that $\mathfrak{a}$ can be written in the required form. This proves~\ref{lem:normalization}.
\end{proof}

\section{The upper bound}

We are now ready to prove our main result.

\begin{proof}[Proof of \Cref{thm:main}]
Fix $\zeta>\frac{41}{2}$. Consider the quadratic polynomials\footnote{Indeed, $f_1$ and $f_2$ have positive leading coefficients if and only if $\zeta>\frac{41}{2}$, which explains the threshold appearing in the statement.}
\[
f_1(k):=\zeta(2k-3)(k-8)-(5k^2-4k)
\quad\text{and}\quad
f_2(k):=\zeta(2k-3)(k-36)-(41k^2-36k).
\]
Now, since both have positive leading coefficient, there exists an integer $K_0=K_0(\zeta)$ such that for every integer $k\ge K_0$ we have
\[
\zeta(2k-3)(k-8)\ge 5k^2-4k
\quad\text{and}\quad
\zeta(2k-3)(k-36)\ge 41k^2-36k.
\]
Set
\[
K:=\max\{K_0,37\}
\qquad\text{and}\qquad
c_{\ref{thm:main}}:=2K.
\]
The lower bound $K\ge 37$ will be used below to ensure that $k-8>0$, $k-36>0$, and $t=\lfloor k/4\rfloor\ge 1$.

Assume that $\Delta\ge c_{\ref{thm:main}}$ and $n\ge \zeta\Delta$. Recall that
$
k=\left\lceil \frac{\Delta}{2}\right\rceil+1,
n=kq+r
$
with $0\le r<k$. Since $\Delta\ge 2K$, we have $k\ge K$. Also, 
it follows that
$
\Delta\ge 2\left\lceil \frac{\Delta}{2}\right\rceil-1=2k-3.
$
Moreover, since $k\le \Delta$ and $n\ge \zeta\Delta>\Delta$, we have $q\ge 1$.
Lemma~\ref{lem:normalization} applies and $\mathfrak{a}$ admits a normalized form. Fix such a representation
$
\mathfrak{a}=xq+u+M,
$
where
$
0\le u\le x\le \left\lfloor \frac{k}{2}\right\rfloor,
u\le r,
r-u\le k-x,$ and $
0\le M<q+k.
$
Choose any set $S\subseteq A$ of size $xq+u$, and let
\[
t:=\left\lfloor \frac{k}{4}\right\rfloor,\qquad
L:=\left\lfloor \frac{\mathfrak{b}-t(q+1)}{\Delta-1}\right\rfloor,\qquad
H:=\min\{q,L\}.
\]
Since $x\le \lfloor k/2\rfloor$, we have
$
t\le \frac{k}{4}\le k-\left\lfloor \frac{k}{2}\right\rfloor\le k-x.$
Together with the normalized form, this gives both
$
r-u\le k-x$ and $
t\le k-x.
$
Thus, in order to apply Lemma~\ref{lem:construction}, it remains only to verify that $H\ge 0$ and $tH\ge M$.

Since $k\ge K\ge 37$, we have
$
t=\left\lfloor \frac{k}{4}\right\rfloor\ge 1.
$
As $H=\min\{q,L\}$, it is enough to prove that $L\ge 0$, $tq\ge M$, and $tL\ge M$.

We claim first that $tq\ge M$. Since
$
q=\left\lfloor \frac{n}{k}\right\rfloor\ge \frac{n}{k}-1$
 and
$t=\left\lfloor \frac{k}{4}\right\rfloor\ge \frac{k}{4}-1,
$
we have
\[
tq\ge \left(\frac{k}{4}-1\right)\left(\frac{n}{k}-1\right)
=\frac{(k-4)(n-k)}{4k}.
\]
On the other hand,
$
M<q+k\le \frac{n}{k}+k.
$
Therefore it is enough to show that
$$
\frac{(k-4)(n-k)}{4k}\ge \frac{n}{k}+k,
$$
or equivalently,
$
n(k-8)\ge 5k^2-4k.
$
Since $\Delta\ge 2k-3$, we have
$
n\ge \zeta\Delta\ge \zeta(2k-3).
$
Also since $k\ge K\ge 37$, it follows that
$
n(k-8)\ge \zeta(2k-3)(k-8).
$
By the choice of $K$, this yields
$
n(k-8)\ge \zeta(2k-3)(k-8)\ge 5k^2-4k.
$
Thus $tq\ge M$.

Next we claim that $tL\ge M$. Note that $\mathfrak{b}\ge n/2$, while
$
t\le \frac{k}{4}$ and
$q+1\le \frac{n}{k}+1.
$
Hence
$
t(q+1)\le \frac{k}{4}\left(\frac{n}{k}+1\right)=\frac{n+k}{4},
$
and so
$
\mathfrak{b}-t(q+1)\ge \frac{n-k}{4}.
$
Consequently,
\[
L\ge \frac{n-k}{4(\Delta-1)}-1\ge \frac{n-9k}{8k},
\]
as $\Delta-1\le 2k$. 

We now note that this lower bound is positive. Indeed, since $k\ge K\ge 37$ and $\zeta>\frac{41}{2}$, we have
$
n\ge \zeta\Delta\ge \zeta(2k-3)>\frac{41}{2}(2k-3)=41k-\frac{123}{2}>9k.
$
Thus
$
\frac{n-9k}{8k}>0,
$
and so in particular $L\ge 0$ and therefore $H\geq 0$. Since also
$
t\ge \frac{k}{4}-1\ge 0,
$
we may multiply the two lower bounds to obtain
\[
tL\ge \left(\frac{k}{4}-1\right)\frac{n-9k}{8k}
=\frac{(k-4)(n-9k)}{32k}.
\]
As above, since
$
M<q+k\le \frac{n}{k}+k,
$
it is enough to prove that
$$
\frac{(k-4)(n-9k)}{32k}\ge \frac{n}{k}+k,
$$
or equivalently,
$
n(k-36)\ge 41k^2-36k.
$
Again, $\Delta\ge 2k-3$ implies
$
n\ge \zeta\Delta\ge \zeta(2k-3)$.
Since $k\ge K\ge 37$, it follows that
$
n(k-36)\ge \zeta(2k-3)(k-36).
$
By the choice of $K$, this gives
$
n(k-36)\ge \zeta(2k-3)(k-36)\ge 41k^2-36k.
$
Thus $tL\ge M$.

We have shown that $H\ge 0$ and $tH\ge M$, so Lemma~\ref{lem:construction} applies. It yields a $(q,q+1)$-cover $\mathcal C$ of $V(G)$ consisting of $k$ independent sets, exactly $r$ of which have size $q+1$. Since $n=kq+r$, this is an equitable $k$-coloring of $G$. This finishes the proof of~\Cref{thm:main}.
\end{proof}

\begin{corollary}\label{cor:algorithmic}
For every $\zeta>\frac{41}{2}$, there exists $c_{\ref{cor:algorithmic}}=c_{\ref{cor:algorithmic}}(\zeta)\in\mathbb{N}$ such that there is an $O(n^2)$-time algorithm which, given a bipartite graph $G$ with $\Delta:=\Delta(G)\ge c_{\ref{cor:algorithmic}}$ and $n\ge \zeta\Delta$, constructs an equitable $\left(\left\lceil \Delta/2\right\rceil+1\right)$-coloring of $G$.
\end{corollary}

\begin{proof}
Let $c_{\ref{cor:algorithmic}}:=c_{\ref{thm:main}}$. We shall show that the proof of \Cref{thm:main} is constructive throughout. First, a bipartition $(A,B)$ of $G$ can be found in time $O(|V(G)|+|E(G)|)$. The proof of Lemma~\ref{lem:normalization} gives an explicit formula for a normalized form $\mathfrak a=xq+u+M$, obtained from a constant number of arithmetic operations, and hence computable in $O(1)$ time. To construct the integers in Observation~\ref{lem:split}, note that in the proof of \Cref{thm:main} we have $t=\lfloor k/4\rfloor\ge 1$ and $M\le tH$; write $M=tQ+R$ with $0\le R<t$, and define $s_i:=Q+1$ for $1\le i\le R$ and $s_i:=Q$ for $R<i\le t$. Then $\sum_{i=1}^t s_i=M$, $0\le s_i\le H$ for all $i\in [t]$, and $|s_i-s_j|\le 1$ for all $i,j\in [t]$, so these integers are computable in $O(t)\le O(n)$ time.

For Lemma~\ref{lem:construction}, the $A$-pure classes and the final $B$-pure classes are obtained by straightforward partitions of vertex sets. The classes $C_1,\dots,C_t$ are built greedily: for each class one chooses the prescribed number of vertices from $A\setminus S$, marks their neighbors in the current set of unused vertices of $B$, and then selects enough unmarked vertices of $B$. The proof of Lemma~\ref{lem:construction} shows that this greedy choice always succeeds. Since the chosen subsets of $A\setminus S$ are pairwise disjoint, every edge incident with a vertex of $A\setminus S$ is scanned at most once, so the total time spent marking neighbors is $O(|E(G)|)$. The search through unused vertices of $B$ costs at most $O(|B|)$ per class, and since there are $t\le k\le n$ such classes, the total cost of this part is $O(n^2)$. All remaining steps take linear time. Therefore, the total running time is $O(n^2+|E(G)|)=O(n^2)$. Correctness follows from \Cref{lem:construction} and \Cref{lem:normalization}.
\end{proof}
\section*{Acknowledgment}
Our thanks to Ararat Harutyunyan for several helpful discussions.

\bibliographystyle{abbrv}
\bibliography{ref}

@article{lih1996equitable,
  title={On equitable coloring of bipartite graphs},
  author={Lih, Ko-Wei and Wu, Pou-Lin},
  journal={Discrete Mathematics},
  volume={151},
  number={1-3},
  pages={155--160},
  year={1996},
  publisher={Elsevier}
}

@article{chen1994-treeequitable,
  title={Equitable coloring of trees},
  author={Chen, Bor-Liang and Lih, Ko-Wei},
  journal={Journal of Combinatorial Theory, Series B},
  volume={61},
  number={1},
  pages={83--87},
  year={1994},
  publisher={Elsevier}
}

@article{chen1994equitable,
  title={Equitable coloring and the maximum degree},
  author={Chen, Bor-Liang and Lih, Ko-Wei and Wu, Pou-Lin},
  journal={European journal of combinatorics},
  volume={15},
  number={5},
  pages={443--447},
  year={1994},
  publisher={Elsevier}
}

@article{erdos1964problem,
  title={Problem 9, Theory of graphs and its applications (M. Fieldler ed.)},
  author={{Erd\"o}s, P},
  journal={Czech. Acad. Sci. Publ., Prague},
  pages={159--159},
  year={1964}
}

@article{hajnal1970proof,
  title={Proof of a conjecture of {P. Erd\"os}},
  author={Hajnal, Andr{\'a}s and Endre Szemer\'edi},
  journal={Combin. Theory Appl.},
  volume={2},
  pages={601},
  year={1970}
}

@article{bollobas1983equitable,
  title={Equitable and proportional coloring of trees},
  author={Bollob{\'a}s, Bela and Guy, Richard K},
  journal={Journal of Combinatorial Theory, Series B},
  volume={34},
  number={2},
  pages={177--186},
  year={1983},
  publisher={Elsevier}
}

\end{document}